\documentclass[leqno,12pt]{amsart} 
\usepackage{amssymb}
\usepackage{enumerate}
\usepackage{color}
\usepackage{graphicx}
\usepackage[normalem]{ulem} 
\usepackage{url}
\theoremstyle{plain} 
\newtheorem{theorem}{\indent\sc Theorem}[section] 

\theoremstyle{definition} 

\newtheorem{remark}[theorem]{\indent\sc Remark}
\newtheorem{example}[theorem]{\indent\sc Example}

\begin{document}

\title[Surfaces of Revolution with Prescribed Mean and Skew Curvatures]{Surfaces of Revolution with Prescribed Mean and Skew Curvatures in Lorentz-Minkowski Space} 

\author[L. C. B. da Silva]{Luiz C. B. da Silva}


\dedicatory{This is a pre-print version of the article published as [Da Silva, {Tohoku Math. J.} \textbf{73} (2021), 317--339] whose full-text is available in \url{https://doi.org/10.2748/tmj.20190729}.}




\address{ 
Departamento de Matem\'atica \endgraf 
Universidade Federal de Pernambuco \endgraf
Recife, PE 50670-901, Brazil.
}
\curraddr{ 
\textsc{Department of Physics of Complex Systems} \endgraf 
\textsc{Weizmann Institute of Science} \endgraf
\textsc{Rehovot 7610001, Israel.}
}
\email{luiz.da-silva@weizmann.ac.il}


\begin{abstract}
In this work, we investigate the problem of finding surfaces in the Lorentz-Minkowski 3-space with prescribed skew {($S$)} and mean {($H$)} curvatures, which are defined through the discriminant of the characteristic polynomial of the shape operator and its trace, respectively. After showing that {$H$} and {$S$} can be interpreted in terms of the expected value and standard deviation of the normal curvature seen as a random variable, we address the problem of prescribed curvatures for  surfaces of revolution. For surfaces with a non-lightlike axis and prescribed {$H$}, the strategy consists in rewriting the equation for {$H$}, which is initially a nonlinear second order Ordinary Differential Equation (ODE), as a linear first order ODE with coefficients in a certain ring of hypercomplex numbers along the generating curves: complex numbers for curves on a spacelike plane and Lorentz numbers for curves on a timelike plane. We also solve the problem for surfaces of revolution with a lightlike axis by using a certain ODE with real coefficients. On the other hand, for the skew curvature problem, we rewrite the equation for {$S$}, which is initially a nonlinear second order ODE, as a linear first order ODE with real coefficients. In all the problems, we are able to find the parameterization for the generating curves in terms of certain integrals of {$H$} and {$S$}.
\end{abstract}

\maketitle

\section*{Introduction}

The problem of finding surfaces with prescribed mean ($H$) or Gaussian ($K$) curvatures is very  important in Differential Geometry. In general, one is led to the study of a non-linear PDE: a nonlinear elliptic PDE of Hessian type  for $K$ \cite{Gutierrez2001,SmithEM2015}, also known as Monge-Amp\`ere equation; and a nonlinear elliptic PDE of divergent type for $H$ \cite{GilbargTrudinger1977}. However, for surfaces invariant by a 1-parameter subgroup of isometries \cite{DoCarmoTohoku1982} this problem is easier and reduces to that of solving a certain non-linear second order ODE \cite{BaikoussisJGeom,DoCarmoTohoku1982,KemmotsuTMJ1981}.  Similar results are also found for surfaces in Lorentz-Minkowski geometry \cite{BenekiJMAA2002,HanoMA1983,HanoTMJ1984,IshiharaJMTU1988} and in other ambient spaces as well \cite{CaddeoMM1995,CaddeoBUMI1996,MontaldoJGP2005,YoonIJM2013}. If we write $H$ and $K$ in terms of the principal curvatures $\kappa_1$ and $\kappa_2$, the problem reduces to finding surfaces with a prescribed sum and product of the principal curvatures. On the other hand, the difference $\kappa_1-\kappa_2$ seems to be given less attention. In this work, we are interested in the study of the mean curvature $H$ and \emph{skew curvature} $S=\sqrt{H^2-\epsilon K}$ in Lorentz-Minkowski space, where the parameter $\epsilon$ is $-1$ for a space-like surface and $+1$ for a time-like one (in Euclidean space $S$ is just $S=\sqrt{H^2-K}$). If the shape operator is diagonalizable, then the skew curvature may be written as $S=\vert \kappa_1-\kappa_2\vert$, while the mean curvature is half the sum $H=\frac{1}{2}(\kappa_1+\kappa_2)$ (in Lorentz-Minkowski space it is possible to have $H^2-\epsilon K<0$, in which case the shape operator is no longer diagonalizable \cite{LopezIEJG2014} and $S$ may be chosen in a way that $\mathrm{i}\,S<0$). 

In the 1960s, B.-Y. Chen obtained global results for Euclidean closed surfaces related with the integrated skew curvature \cite{ChenMJOU1969} (named by him as {the} \emph{difference curvature}). In the {1970s}, the skew curvature was independently reintroduced in Euclidean space by T. K. Milnor as an auxiliary tool in the study of open surfaces \cite{MilnorJDG1976}. Later, she investigated properties of some quadratic forms defined by using the skew curvature \cite{MilnorPAMS1977}. It is worth mentioning that at an umbilic point the skew curvature vanishes and, consequently, it can work as a measure of the surface bend anisotropy. Indeed, we shall see that the mean and skew curvatures can be associated with the expected value and standard deviation of the normal curvature seen as a random variable, see Eq. (\ref{eq::StatInterprationSandH}). In addition, the behavior of $S$ can be related {to} the diagonalizability of the shape operator, a problem that does make sense in non-Riemannian geometry { \cite{Clelland2017,DaSilvaArXiv2018,FujiokaASPM2008,LopezIEJG2014} }. Moreover, in {a 3-dimensional space form of curvature $c$} it is valid the relation $H^2-K\geq-c$, with  equality valid for totally umbilical surfaces only \cite{ChenGMJ1996}. {(A similar expression holds in semi-Riemannian space forms as well. However, for timelike surfaces, i.e., $\epsilon=+1$, the equality $H^2-K=-c$ does not imply umbilicity in general \cite{FujiokaASPM2008}.)} Finally, let us mention that the square of the skew curvature appears in the study of the Willmore functional $W=\int(H^2-K)\mathrm{d}{A}$ \cite{Willmore1992} and also as a geometry-induced potential in the context of the quantum dynamics of a particle constrained to move on a surface in Euclidean space \cite{DaCostaPRA1981,DaSilvaAP2017}. 

Recently, surfaces in Euclidean space
with constant skew curvature were investigated \cite{LiDGA2015,TodaCSkC} and the problem of finding surfaces of revolution with prescribed skew curvature was solved in the context of a quantum constrained dynamics { \cite{DaSilvaAP2017} }. In this work, we shall address the problem of finding surfaces of revolution with prescribed skew curvature in Lorentz-Minkowski space following similar techniques to those of Ref. \cite{DaSilvaAP2017}. In addition, we also revisit the problem of finding Lorentzian surfaces of revolution with prescribed mean curvature by using the complex numbers and the so-called Lorentz numbers (also known as double numbers, see supplement C of Ref. \cite{Yaglom1979}), which constitutes a natural generalization of the technique employed in Euclidean space \cite{KemmotsuTMJ1981} and can allow for a better understanding of the results found in Lorentz-Minkowski space \cite{HanoTMJ1984,IshiharaJMTU1988}. 

This work is divided as follows. In Section 1 we present the fundamentals of the geometry in Lorentz-Minkowski space along with the classification of rotations, an Euler theorem for the normal curvature $\kappa_n$, and a statistical interpretation of $H$ and $S$ as an expected value and standard deviation of $\kappa_n$, respectively. In Section 2 we address the problem of finding surfaces of revolution with prescribed mean curvature: the surfaces with a non-lightlike axis are described in Subsections 2.1., 2.2., and 2.3, whose solution of the prescribed $H$ problem is analyzed in 2.4; and, in Subsection 2.5, we solve the problem for surfaces of revolution with a lightlike axis. Finally, in Section 3, we solve the   prescribed skew curvature problem for surfaces of revolution with a non-lightlike axis. In Appendix A, we present the ring of Lorentz numbers which constitute an important tool in Section 2. 

The present author would like to thank useful discussions with Renato T. Gomes (from Universidade Federal Rural de Pernambuco, Recife, Brazil).

\section{Differential geometric background}

We now present some geometric preliminaries and also establish an Euler theorem for the normal curvature. {This leads} to a statistical interpretation for $H$ and $S$ which qualifies them as appropriate quantities in the study of the extrinsic behavior of a surface. In later sections,  we shall study all the basic types of surfaces of revolution in $\mathbb{E}_1^3$ (see Table 1) and show how to find surfaces of revolution with prescribed mean or skew curvature. Both problems shall be solved by conveniently rewriting the respective curvature equations in terms of certain linear ODE's.

Let us denote by $\mathbb{E}_1^3$ the 3-dimensional Lorentz-Minkowski space, i.e., the vector space $\mathbb{R}^3$ equipped with the index one metric
\begin{equation}
\Big\langle (u_1,u_2,u_3),(v_1,v_2,v_3)\Big\rangle = u_1v_1+u_2v_2-u_3v_3\,.
\end{equation}
On the other hand, we shall denote the usual Euclidean space by $\mathbb{E}^3$.

In $\mathbb{E}_1^3$ we may introduce the concept of \emph{causal character} as follows: we say that $v\in\mathbb{E}_1^3$ is (i) \emph{spacelike}, (ii) \emph{timelike}, or (iii) \emph{lightlike} if (i) $\langle v,v\rangle>0$ or $v=0$, (ii) $\langle v,v\rangle<0$, or (iii) $\langle v,v\rangle=0$ and $v\not=0$, respectively. Given a regular parameterized curve $\alpha:I\to\mathbb{E}_1^3$, i.e., $\alpha'\not=0$, we say that $\alpha$ is a \emph{spacelike}, \emph{timelike}, or \emph{lightlike} curve if $\alpha'$ is spacelike, timelike, or lightlike in $I$, respectively (for space- or time-like curves we may introduce an arc-length parameter $s$ as usual: $s=\int\sqrt{\vert\langle\alpha'(t),\alpha'(t)\rangle\vert}\,\mathrm{d}t$). On the other hand, for a regular surface $\Sigma$ given by a parameterization $X:U\to\Sigma\subset\mathbb{E}_1^3$, we say that $\Sigma$ is a \emph{spacelike}, \emph{timelike}, or \emph{lightlike} surface if the induced metric $p\mapsto\langle\cdot,\cdot\rangle\vert_{T_p\Sigma}$ is Riemannian, Lorentzian (non-degenerate with index 1), or degenerate with rank 1, respectively.

A \emph{rotation} in $\mathbb{E}_1^3$ is an isometry leaving a certain straight line pointwise fixed, known as {the} \emph{rotation axis}. Rotations are completely determined by the causal character of the respective rotation axis \cite{LopezIEJG2014}. In this way, it suffices to consider the three cases below:
\begin{enumerate}[(a)]
\item \textit{timelike axis}: supposing that the axis is $(0,0,1)$, we have
\begin{equation}
T_{\theta} = \left(
\begin{array}{ccc}
\cos\theta & -\sin\theta & 0\\
\sin\theta & \cos\theta & 0\\
0 & 0 & 1\\
\end{array}
\right),\, \theta\in\mathbb{S}^1;
\end{equation}
\item \textit{spacelike axis}: supposing that the axis is $(1,0,0)$, we have
\begin{equation}
S_{\theta} = \left(
\begin{array}{ccc}
1 & 0 & 0 \\
0 & \cosh\theta & \sinh\theta\\
0 & \sinh\theta & \cosh\theta\\
\end{array}
\right),\,\theta\in\mathbb{R}\,; \mbox{ and }
\end{equation}
\item \textit{lightlike axis}: supposing that the axis is $(0,1,1)$, we have
\begin{equation}
L_{\theta} = \left(
\begin{array}{ccc}
1 & \theta & -\theta \\
-\theta & 1-\frac{\theta^2}{2} & \frac{\theta^2}{2}\\[5pt]
-\theta & -\frac{\theta^2}{2} & 1+\frac{\theta^2}{2}\\
\end{array}
\right),\,\theta\in\mathbb{R}\,.
\end{equation}
\end{enumerate}

Apart from subsection 1.1, in this work we shall be interested in surfaces of revolution only, i.e., surfaces $\Sigma$ invariant by $T_{\theta}$, $S_{\theta}$, or $L_{\theta}$: e.g., for a timelike axis, $\Sigma=T_{\theta}(\Sigma)$ for all $\theta$. Here, the whole surface can be obtained by rotating a curve, the \emph{generating curve}, that can be assumed to be contained in a plane (which also contains the axis, Figure 1).

The many possibilities for the causal characters of the rotation axis in combination with the causal characters of the generating curve and the plane where it is contained in give rise to various types of surfaces of revolution in $\mathbb{E}_1^3$, as will become clear in the following (see Table 1 and Figure 1).

\begin{table}[h]
\centering
\large
\begin{tabular}{|p{2cm}|p{2cm}|p{2cm}|p{2cm}|}
\hline
Axis & Plane  & Curve & Surface  \\ \hline \hline
time  & time  & time  & time  \\ \cline{3-4}
      &       & space & space \\ \hline 
space & time  & time  & time  \\ \cline{3-4}
      &       & space & space \\ \cline{2-4}
      & space & space & time \\ \hline
light & time  & time  & time  \\ \cline{3-4}
      &       & space & space \\ \hline
\end{tabular}
\vspace{10pt}
\caption{Causal characters of surfaces of revolution in terms of the causal characters of the rotation axis, the plane that contains the generating curve, and the generating curve (the only lightlike surfaces of revolution are lightlike planes and light cones \cite{InoguchiIJGMMP2009}).}
\label{tab:adicaoZ4}
\end{table}

Due to the intimate relationship between the causal characters of a vector subspace $V\subset\mathbb{E}_1^3$ and its orthogonal complement $V^{\perp}$ induced by $\langle\cdot,\cdot\rangle$, Prop. 1.1 of \cite{LopezIEJG2014}, if a surface $\Sigma$ admits a unit normal vector field $N$ (in particular, $\Sigma$ is not lightlike), then we have
\begin{equation}
\epsilon = \langle N,N\rangle\Rightarrow \epsilon=\left\{
\begin{array}{cl}
-1\,, &\mbox{ if } { \Sigma }   \mbox{ is spacelike}\\
+1\,, &\mbox{ if } { \Sigma }   \mbox{ is timelike}\\
\end{array}
\right..
\end{equation}
In local coordinates $X:U\subset\mathbb{R}^2\to \Sigma\subset\mathbb{E}_1^3$, $X=X(u,v)$, the normal can be written as
\begin{equation}
N = \frac{X_u\times X_v}{\sqrt{\vert X_u\times X_v\vert}},
\end{equation}
where $\times$ is the cross product in $\mathbb{E}_1^3$: 
\begin{equation} 
(u_1,u_2,u_3)\times(v_1,v_2,v_3)=(u_2v_3-u_3v_2,-(u_1v_3-u_3v_1),-(u_1v_2-u_2v_1)).
\end{equation}

The coefficients $g_{ij}$ and $h_{ij}$ of the \emph{first} and \emph{second fundamental forms} of $\Sigma$ are defined as 
$g_{11} = \langle X_u,X_u\rangle,\, g_{12} = \langle X_u,X_v\rangle,\,g_{22} = \langle X_v,X_v\rangle$, 
and $h_{11} = \langle X_{uu},N\rangle,\, h_{12} = \langle X_{uv},N\rangle,\,h_{22} = \langle X_{vv},N\rangle$, respectively. It is worth mentioning that $-\epsilon=\mbox{sgn}(g):=\mbox{sgn}(\det\,g_{ij})$ and, therefore, $g>0$ for a spacelike surface and $g<0$ for a timelike one (if $\Sigma$ is lightlike, then $g=0$) \cite{LopezIEJG2014}. 

Finally, in local coordinates the \emph{Gaussian} $K$ and \emph{mean} $H$ \emph{curvatures} are \cite{LopezIEJG2014}
\begin{equation}
K=\epsilon\,\frac{h_{11}h_{22}-h_{12}^2}{g_{11}g_{22}-g_{12}^2} \mbox{ and }H=\frac{\epsilon}{2}\,\frac{g_{11}h_{22}-2g_{12}h_{12}+g_{22}h_{11}}{g_{11}g_{22}-g_{12}^2},
\end{equation}
while the \emph{skew curvature} $S$ is
\begin{equation}
S = \sqrt{H^2-\epsilon K}.
\end{equation}
By convention, if $H^2-\epsilon K<0$, we may choose $S$ in a way that $\mathrm{i}\,S<0$. {For surfaces of revolution, however, we do not need to worry about such a possibility. Indeed, since the shape operator $A_p=-\mathrm{d}N$ is always diagonalizable for surfaces of revolution \cite{FujiokaASPM2008} (see also the explicit computations in Sect. \ref{sec::PrescMC}) and since the discriminant of the characteristic polynomial of $A_p$ is precisely $4(H^2-\epsilon K)$ \cite{LopezIEJG2014}, it follows that $S^2=H^2-\epsilon K\geq0$ here.}
\begin{remark}
In \cite{Clelland2017}, the skew curvature is defined to be $\sqrt{H^2-K}$ and denoted by $H'$. Here, we shall denoted it by $S$ (from skew) instead of $H'$ in order to avoid confusion with {the} derivative of $H$. In addition, since $4(H^2-\epsilon K)$ is the discriminant of the characteristic polynomial of the shape operator \cite{LopezIEJG2014}, it seems to be more natural to define $S$ the way we do. Finally, we believe that the results to be presented in the subsection below and also in Section 3, for {surfaces of revolution} with non-lightlike axis, will show that our definition allows for a suitable use of the skew curvature concept.
\end{remark}

\subsection{Euler theorem and statistical interpretation of the mean and skew curvatures}

Since $S$ vanishes at an umbilic point, it can be thought to be a measure of the surface bend anisotropy. Indeed, we shall show below that the mean and skew curvatures are respectively given in terms of the expected value and standard deviation of the normal curvature $\kappa_n$, when we see $\kappa_n$ as a random variable. This suggests that $H$ and $S$ together are appropriate quantities to {offer} a glimpse of the extrinsic behavior of a surface.

Let $\Sigma$ be a regular parameterized surface, not necessarily of revolution, and $p\in\Sigma$. The normal curvature $\kappa_n$ at $p$ is a real function over the set of unit tangent vectors, i.e., $\kappa_n:\mathbb{S}^1 { \subset T_p\Sigma } \to\mathbb{R}$ for a spacelike surface or $\kappa_n:\mathbb{S}_1^1\cup\mathbb{H}^1 { \subset T_p\Sigma }\to\mathbb{R}$ for a timelike one. We assume, in addition, that the shape operator $A_p=-\mathrm{d}N$ at $p\in\Sigma$ has non-lightlike eigenvectors $\mathbf{u}_i$ ($i=1,2$)  with eigenvalues $\kappa_i$ ($i=1,2$). In particular, we are assuming that $A_p$ is diagonalizable, in which case the mean and skew curvatures can be written as $H=\frac{\epsilon}{2}(\kappa_1+\kappa_2)$ and $S=\sqrt{(\kappa_1-\kappa_2)^2}$. (This is the case for {surfaces of revolution}.) 

If $\Sigma$ is spacelike, the induced metric is Riemannian and then we can write any unit tangent vector $\mathbf{v}$ at $p$ as
\begin{equation}
\mathbf{v} = \cos\phi\,\mathbf{u}_1+\sin\phi\,\mathbf{u}_2,\, \phi\in\mathbb{S}^1.
\end{equation}
This leads to the following Euler theorem
\begin{equation}
\kappa_n(p,\mathbf{v})=\langle A_p\,\mathbf{v},\mathbf{v}\rangle=\cos^2(\phi)\,\kappa_1+\sin^2(\phi)\,\kappa_2.\label{eq::EulerFormulaSpacelikeSurf}
\end{equation}
Notice this is the same expression we would obtain for a surface $\Sigma$ in Euclidean space.

On the other hand, if $\Sigma$ is timelike, the induced metric is Lorentzian and then we can write any unit tangent vector $\mathbf{v}$ at $p$ as
\begin{equation}
\mathbf{v} = \left\{
\begin{array}{r}
\pm\cosh\phi\,\mathbf{u}_1+\sinh\phi\,\mathbf{u}_2, \mbox{ if }\langle\mathbf{v},\mathbf{v}\rangle=+1\\
\sinh\phi\,\mathbf{u}_1\pm\cosh\phi\,\mathbf{u}_2,\mbox{ if }\langle\mathbf{v},\mathbf{v}\rangle=-1\\
\end{array}
\right.,\,\phi\in\mathbb{R},
\end{equation}
where we are assuming for simplicity that $\mathbf{u}_1$ is the spacelike eigenvector and $\mathbf{u}_2$ is the timelike one. 
This leads to the following Euler theorem
\begin{equation}
\kappa_n(p,\mathbf{v})=\left\{
\begin{array}{c}
\cosh^2(\phi)\,\kappa_1-\sinh^2(\phi)\,\kappa_2, \mbox{ if }\langle\mathbf{v},\mathbf{v}\rangle=+1\\
\sinh^2(\phi)\,\kappa_1-\cosh^2(\phi)\,\kappa_2,\mbox{ if }\langle\mathbf{v},\mathbf{v}\rangle=-1\\
\end{array}
\right..\label{eq::EulerFormulaTimelikeSurf}
\end{equation}

Now, pretending $\kappa_n$ is a random variable, there  are two important parameters naturally associated with it, namely the expected value ${\langle\kappa_n\rangle}$ and the standard deviation ${\sqrt{\langle(\Delta\kappa_n)^2\rangle}}$. Then, if $\Sigma$ is spacelike ($\epsilon=-1$),  we can use Eq. (\ref{eq::EulerFormulaSpacelikeSurf}) to establish the following relation for the expected value of $\kappa_n$ with respect to the uniform (probability) density $\frac{\mathrm{d}\phi}{2\pi}$ 
\begin{equation}
{\langle\kappa_n\rangle} = \int_{0}^{2\pi}\kappa_n(\phi)\,\frac{\mathrm{d}\phi}{2\pi}=\frac{\kappa_1+\kappa_2}{2}\Rightarrow H=-{\langle\kappa_n\rangle}. 
\end{equation}
In addition, the standard deviation is
\begin{equation}
{\sqrt{\langle(\Delta\kappa_n)^2\rangle}}=\sqrt{\int_{0}^{2\pi}[\kappa_n(\phi)-{\langle\kappa_n\rangle}]^2\,\frac{\mathrm{d}\phi}{2\pi}}=\sqrt{\frac{(\kappa_1-\kappa_2)^2}{8}}\Rightarrow S =2\sqrt{2}\,{\sqrt{\langle(\Delta\kappa_n)^2\rangle}}. 
\end{equation}
It is worth mentioning that in Euclidean space we would find analogous results for $\kappa_n$, namely ${\langle\kappa_n\rangle}=H$ and $S=2\sqrt{2}\,{\sqrt{\langle(\Delta\kappa_n)^2\rangle}}$, as can be easily verified.

The results above for surfaces with an induced metric of Riemannian signature suggest that we can replace the continuous distribution $\kappa_n$ by a discrete one taking the possible values  $\kappa_1$ or $\kappa_2$, since in this case
\begin{equation}
{\langle\kappa_n\rangle}=\frac{1}{2}\sum_{i=1}^2\kappa_i=\frac{\kappa_1+\kappa_2}{2}
{ \mbox{ and } }
{\sqrt{\langle(\Delta\kappa_n)^2\rangle}}=\sqrt{\frac{1}{2-1}\sum_{i=1}^2(\kappa_i-{\langle\kappa_n\rangle})^2}=\sqrt{\frac{1}{2}(\kappa_1-\kappa_2)^2}\,.
\end{equation}

Now, let $\Sigma$ be a timelike surface ($\epsilon=+1$). Due to symmetry considerations, we restrict ourselves to what happens along a single branch of the pair of hyperbolas associated with $\mathbb{S}_1^1\cup\mathbb{H}^1$, say the branch parameterized by $\phi\mapsto(\cosh\phi,\sinh\phi)$. Observe that, unlike the spacelike case, here the integrals $\int_{\mathbb{R}}\cosh^2\phi\,\mathrm{d}\phi$ and $\int_{\mathbb{R}}\sinh^2\phi\,\mathrm{d}\phi$ do not converge and, consequently, $\kappa_n$ fails to have a finite expected value. {The averages may diverge even if we divide $\int_{-a}^af(\phi)\mathrm{d}\phi$ by the length of the region of integration, $2a$, and take $a\to\infty$. To obtain something meaningful, we may consider only the finite part of the limit, i.e., define} 
\begin{equation}
   \langle\kappa_n\rangle = \mbox{finite part}\left(\lim_{a\to\infty}\frac{1}{2a}\int_{-a}^{+a}\kappa_n(\phi)\mathrm{d}\phi\right).
\end{equation}

{We have}
\begin{equation}
    \left\langle\left(\frac{\mathrm{e}^{\phi}\pm\mathrm{e}^{-\phi}}{2}\right)^2\right\rangle_a = \frac{1}{2a}\int_{-a}^{+a}\frac{\mathrm{e}^{2\phi}\pm2+\mathrm{e}^{-2\phi}}{4}\mathrm{d}\phi=\pm\frac{1}{2}+\frac{\sinh(2a)}{4a}\stackrel{a\gg1}{\longrightarrow}\pm\frac{1}{2}+\frac{\mathrm{e}^{2a}}{4a}.
\end{equation}
{Taking into account the finite contributions only, the expected value of $\kappa_n$ is given by}
\begin{equation}
    \langle\kappa_n\rangle=\langle\kappa_1\cosh^2\phi-\kappa_2\sinh^2\phi\rangle=\frac{\kappa_1+\kappa_2}{2}=H.
\end{equation}

{On the other hand, to find the standard deviation, we first compute}
\begin{equation}
    \left\langle\left(\frac{\mathrm{e}^{\phi}\pm\mathrm{e}^{-\phi}}{2}\right)^4\right\rangle_a = \frac{1}{2a}\int_{-a}^{+a}\left(\pm\frac{1}{2}+\frac{\cosh2\phi}{2}\right)^2\mathrm{d}\phi=\frac{3}{8}+\frac{\sinh2a}{4a}+\frac{\sinh4a}{32a}
\end{equation}
{and}
\begin{equation}
    \left\langle\cosh^2\phi\sinh^2\phi\right\rangle_a = \frac{1}{2a}\int_{-a}^{+a}\left(\frac{1}{2}+\frac{\cosh2\phi}{2}\right)\left(-\frac{1}{2}+\frac{\cosh2\phi}{2}\right)\mathrm{d}\phi=-\frac{1}{8}+\frac{\sinh4a}{32a}.
\end{equation}
{Taking into account the finite contributions only,}
\begin{equation}
\langle(\Delta\kappa_n)^2\rangle= \langle\kappa_1^2\cosh^4\phi-2\kappa_1\kappa_2\cosh^2\phi\sinh^2\phi+\kappa_2^2\sinh^4\phi\rangle-\left(\frac{\kappa_1+\kappa_2}{2}\right)^2=\frac{1}{8}(\kappa_1-\kappa_2)^2.
\end{equation}
{Finally, the standard deviation of the normal curvature is given by}
\begin{equation}
    \sqrt{\langle(\Delta\kappa_n)^2\rangle} = \sqrt{\frac{(\kappa_1-\kappa_2)^2}{8}}=\frac{S}{2\sqrt{2}}.
\end{equation}

In short, the results above for both space- and time-like surfaces {give} that the expected value ${\langle\kappa_n\rangle}$ and the standard deviation ${\sqrt{\langle(\Delta\kappa_n)^2\rangle}}$ of $\kappa_n$ are associated with the mean $H$ and skew $S$ curvatures according to
\begin{equation}
H=\epsilon\,{\langle\kappa_n\rangle \mbox{ and }}S=2\sqrt{2}\,{\sqrt{\langle(\Delta\kappa_n)^2\rangle}}.\label{eq::StatInterprationSandH}
\end{equation}
\begin{remark}
{A similar procedure of considering the finite contributions of diverging integrals also appears in the formulation of the Cauchy integral formula for functions over the Lorentz numbers \cite{CatoniAACA2012}. Indeed, in general $f(w)\not=\frac{1}{2\pi \tau}\int_{\gamma}\frac{f(z)}{z-w}\mathrm{d}z$, where $\gamma(\theta)=w+R\mathrm{e}^{\mathrm{i}\theta}$ is an Euclidean circle around $z=w$. However, considering only the finite contribution of the same integral over the branch of a hyperbola we obtain  $f(w)=\lim_{a\to\infty}\frac{1}{2a \tau}\int_{\gamma}\frac{f(z)}{z-w}\mathrm{d}z$.}
\end{remark}

\section{Prescribed mean curvature equation in Lorentz-Minkowski space}
\label{sec::PrescMC}

In this section we solve the problem of prescribed mean curvature. Following Kenmotsu \cite{KemmotsuTMJ1981}, the strategy for {surfaces of revolution} with a non-lightlike axis (subsections 2.1, 2.2, and 2.3) consists in considering the generating curve parameterized by arc-length and then write the equation for the mean curvature, which is initially a nonlinear second order ODE, as a linear first order ODE with coefficients in a certain ring of hypercomplex numbers along the generating curves (subsection 2.4): complex number $\mathbb{C}$ for curves on a spacelike plane and Lorentz numbers $\mathbb{L}$ (see Appendix A) for curves on a timelike plane. For a lightlike axis we are still able to solve the prescribed $H$ problem using the real numbers $\mathbb{R}$ (subsection 2.5).

\begin{remark}
{The surfaces described in section \ref{sec::MCforCurvOnXYplaneWithOXaxis}, i.e., the ones generated from curves on a spacelike plane rotated around a spacelike axis (Figure 1(c)), furnish a counter-example} to the assertion that a surface of revolution in $\mathbb{E}_1^3$ inherits the causal character of its generating curve \cite{IshiharaJMTU1988} (there Ishihara and Hara only take into account the revolution of curves on the timelike $yz$-plane).
\end{remark}

\begin{figure*}[tbp]
\centering
    {\includegraphics[width=0.24\linewidth]{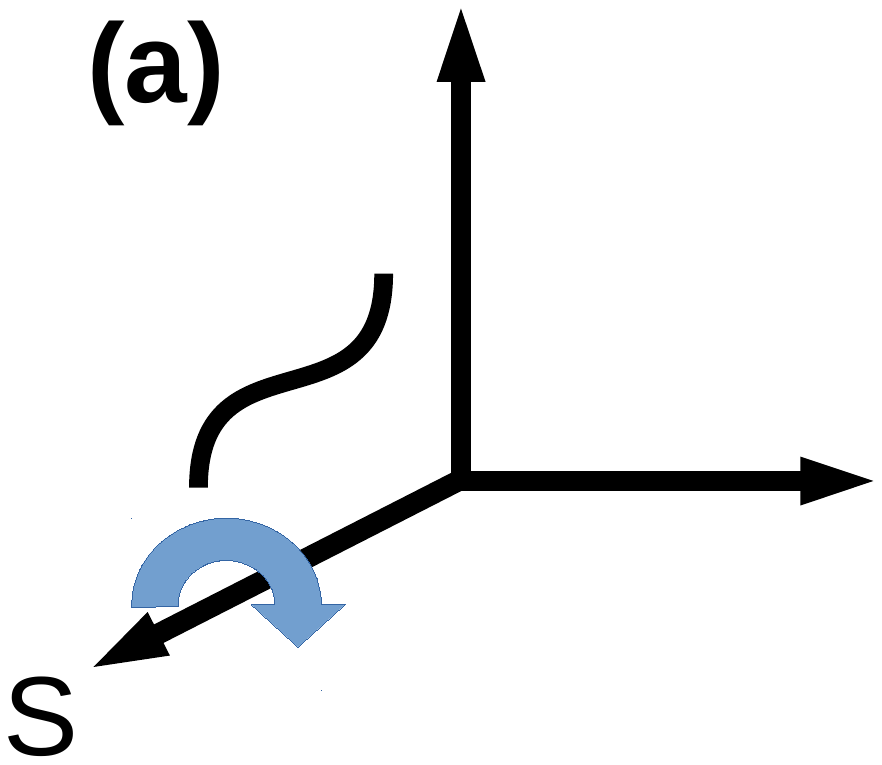}}
    {\includegraphics[width=0.24\linewidth]{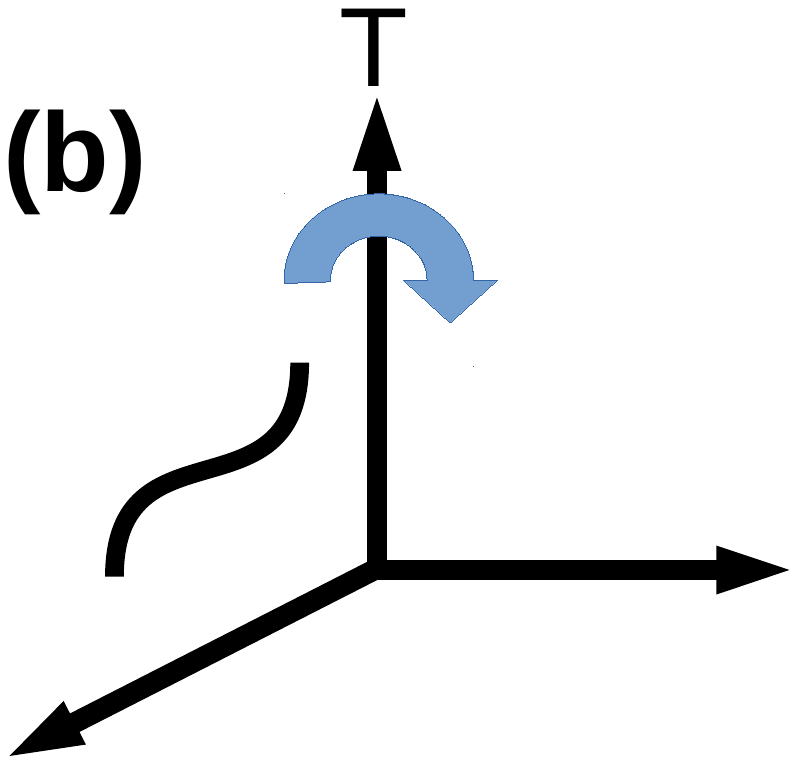}}
    {\includegraphics[width=0.24\linewidth]{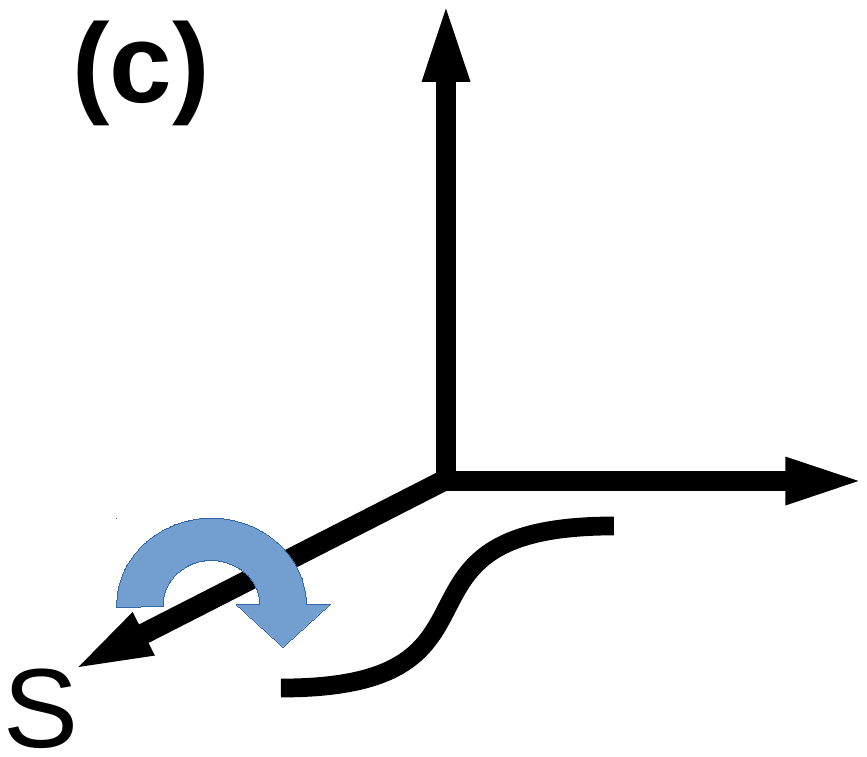}}
    {\includegraphics[width=0.24\linewidth]{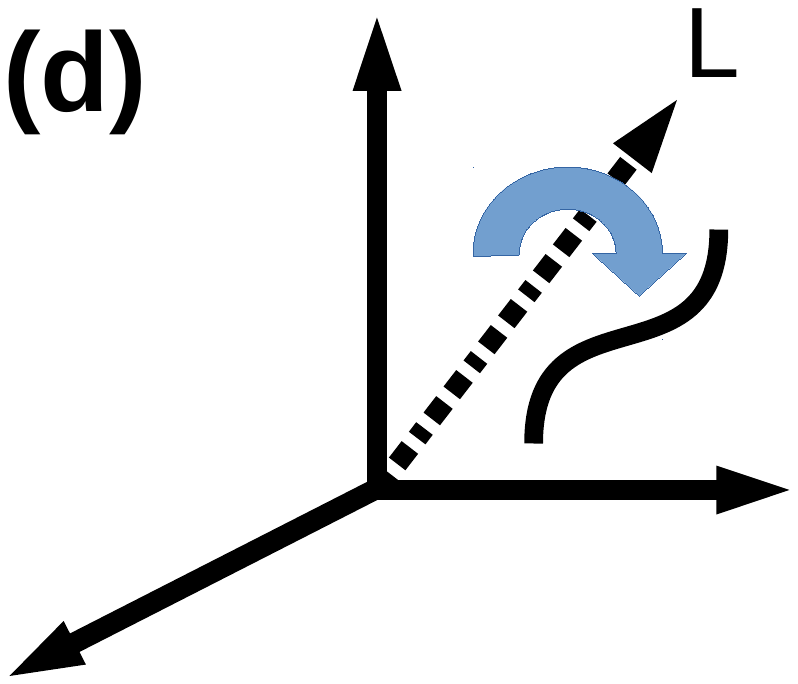}}

\caption{ { Surfaces of revolution} in $\mathbb{E}_1^3$: (a) for curves on a timelike plane with a spacelike axis {S} there are two types: a space- or time-like surface for a space- or time-like curve, respectively; (b) for  curves on a timelike plane with a timelike axis {T} there are two types: a space- or time-like surface for a space- or time-like curve, respectively; (c) for  curves on a spacelike plane and, consequently, with a spacelike axis {S} there is one type: a timelike surface; and (d) for  curves on a timelike plane with a lightlike axis {L} there are two types: a space- or time-like surface for a space- or time-like curve, respectively.}
          \label{fig::DiagramSphPlaneCurv}
\end{figure*}

\subsection{Rotation of a curve on a timelike plane around a timelike axis}

Let $\alpha:I\to \mathbb{E}_1^3$ be a $C^2$ regular curve in the $xz$-plane, i.e., $\alpha(s)=(x(s),0,z(s))$ with $s$ arc-length parameter and $x>0$. Considering a rotation of this curve around the $z$-axis gives the following surface of revolution
\begin{equation}
Z(s,\theta) = (x(s)\cos\theta,x(s)\sin\theta,z(s)),\label{defRevSurfTimeAxisXZcurve}
\end{equation}
where $\theta\in(0,2\pi)$. Since $s$ is the arc-length parameter of $\alpha$, we can write
\begin{equation}
\eta = \langle\alpha',\alpha'\rangle = x'\,^2-z'\,^2\in\{-1,1\}.
\end{equation}

The first fundamental form $\mathrm{I}$ is given by
\begin{equation}
\mathrm{I} = \eta\, {\rm d}s^2+x^2\,{\rm d}\theta^2\,.\label{eq1stFFormRevSurfTimeAxisXZcurve}
\end{equation}
Since $g_{11}g_{22}-g_{12}^2=\eta\,x^2\Rightarrow \epsilon=-\eta$, it follows that $Z$ is a spacelike (timelike) surface if and only if $\alpha$ is a spacelike (timelike) curve.

Writing the normal vector to $Z$ as
\begin{equation}
N = (-z'\,\cos\theta,-z'\,\sin\theta,-x'),
\end{equation}
the second fundamental form $\mathrm{II}$ is given by
\begin{equation}
\mathrm{II} = (x'z''-x''z')\,{\rm d}s^2+xz'\,{\rm d}\theta^2\,.\label{eq2ndFFormRevSurfTimeAxisXZcurve}
\end{equation}
Since both I and II are diagonal, the shape operator $A=\mathrm{I}^{-1}\mathrm{II}$ is diagonalizable.

The mean curvature equation is then written as
\begin{equation}
2xH+xx'z''-xx''z'+\eta z'=0\,.\label{eqMeanCurvRevSurfTimeAxisXZcurve}
\end{equation}
Since $\alpha$ is parametrized by arc-length, we have the additional equation
\begin{equation}
x'\,^2-z'\,^2 = \eta 
{ \,\Rightarrow\, }
x'x'' = z'z''.\label{eqAddRelationsRevSurfTimeAxisXZcurve}
\end{equation}

Multiplying Eq. (\ref{eqMeanCurvRevSurfTimeAxisXZcurve}) by $x'$ gives
\begin{equation}
2xx'H+\eta(xz')'=0.
\end{equation}
On the other hand, multiplying Eq. (\ref{eqMeanCurvRevSurfTimeAxisXZcurve}) by $z'$ gives
\begin{equation}
2xz'H+\eta\,(xx')'-1=0.
\end{equation}
By defining $A(s)=x(s)x'(s)+\tau\,x(s)z'(s)$ in $\mathbb{L}$, we can write the two equations above in a single expression as
\begin{equation}
A'(s)+2\,\tau\,\eta\, H(s) \,A(s)-\eta = 0\,.\label{eqMeanCurvWithHyperbolicNumRevSurfTimeAxisXZcurve}
\end{equation}

\subsection{Rotation of a curve on a timelike plane around a spacelike axis}

Let $\beta:I\to \mathbb{E}_1^3$ be a $C^2$ regular curve in the $xz$-plane, i.e., $\alpha(s)=(x(s),0,z(s))$ with $s$ arc-length parameter and $z>0$. Considering a rotation of this curve around the $x$-axis gives the following surface of revolution
\begin{equation}
X_I(s,\theta) = (x(s),z(s)\sinh\theta,z(s)\cosh\theta),\label{defRevSurfSpaceAxisXZcurve}
\end{equation}
where $\theta\in(-\infty,+\infty)$. Since $s$ is the arc-length parameter of $\beta$, we can write
\begin{equation}
\eta = \langle\alpha',\alpha'\rangle = x'\,^2-z'\,^2\in\{-1,1\}.
\end{equation}

The first fundamental form I is given by
\begin{equation}
\mathrm{I} = \eta\, {\rm d}s^2+z^2\,{\rm d}\theta^2\,.\label{eq1stFFormRevSurfSpaceAxisXZcurve}
\end{equation}
Since $g=g_{11}g_{22}-g_{12}^2=\eta\,z^2\Rightarrow\epsilon=-\eta$, it follows that $X_I$ is a spacelike (timelike) surface if and only if $\beta$ is a spacelike (timelike) curve.

Writing the normal vector to $X_I$ as
\begin{equation}
N = (-z',-x'\sinh\theta,-x'\,\cosh\theta),
\end{equation}
the second fundamental form II is given by
\begin{equation}
\mathrm{II} = (x'z''-x''z')\,{\rm d}s^2+zx'\,{\rm d}\theta^2\,.\label{eq2ndFFormRevSurfSpaceAxisXZcurve}
\end{equation}
Since both I and II are diagonal, the shape operator $A=\mathrm{I}^{-1}\mathrm{II}$ is diagonalizable.

The mean curvature equation is then written as
\begin{equation}
2zH+zx'z''-zx''z'+\eta\, x'=0\,.\label{eqMeanCurvRevSurfSpaceAxisXZcurve}
\end{equation}
Since $\beta$ is parameterized by arc-length, we have the additional equation
\begin{equation}
x'\,^2-z'\,^2 = \eta
{ \,\Rightarrow\, }
x'x'' = z'z''.\label{eqAddRelationsRevSurfSpaceAxisXZcurve}
\end{equation}

Multiplying Eq. (\ref{eqMeanCurvRevSurfSpaceAxisXZcurve}) by $x'$ gives
\begin{equation}
2zx'H+\eta(zz')'+1=0.
\end{equation}
On the other hand, multiplying Eq. (\ref{eqMeanCurvRevSurfSpaceAxisXZcurve}) by $z'$ gives
\begin{equation}
2zz'H+\eta(zx')'=0.
\end{equation}
By defining $B(s)=z(s)z'(s)+\tau\,z(s)x'(s)$ in $\mathbb{L}$, we can write the two equations above in a single expression as
\begin{equation}
B'(s)+2\,\tau\,\eta \,H(s) \,B(s)+\eta = 0\,.\label{eqMeanCurvWithHyperbolicNumRevSurfSpaceAxisXZcurve}
\end{equation}

\subsection{Rotation of a curve on a spacelike plane around a spacelike axis}
\label{sec::MCforCurvOnXYplaneWithOXaxis}

Let $\gamma:I\to \mathbb{E}_1^3$ be a $C^2$ regular curve in the $xy$-plane, i.e., $\gamma(s)=(x(s),y(s),0)$ with $s$ arc-length and $y>0$. Considering a rotation of this curve around the $x$-axis gives the following surface of revolution
\begin{equation}
X_{II}(s,\theta) = (x(s),y(s)\cosh\theta,y(s)\sinh\theta),\label{defRevSurfSpaceAxisXYcurve}
\end{equation}
where $\theta\in(-\infty,+\infty)$. Since $s$ is the arc-length parameter of $\gamma$, the first fundamental form I is given by
\begin{equation}
\mathrm{I} = {\rm d}s^2-y^2\,{\rm d}\theta^2\,.\label{eq1stFFormRevSurfSpaceAxisXYcurve}
\end{equation}
Since $g=g_{11}g_{22}-g_{12}^2=-y^2\Rightarrow\epsilon=+1$, it follows that $X_{II}$ is a timelike surface (observe that $\gamma$ is necessarily a spacelike curve).

Writing the normal vector to $X_{II}$ as
\begin{equation}
N = (y',-x'\cosh\theta,-x'\,\sinh\theta),
\end{equation}
the second fundamental form II is given by
\begin{equation}
\mathrm{II} = (x''y'-x'y'')\,{\rm d}s^2-x'y\,{\rm d}\theta^2\,.\label{eq2ndFFormRevSurfSpaceAxisXYcurve}
\end{equation}
Since both I and II are diagonal, the shape operator $A=\mathrm{I}^{-1}\mathrm{II}$ is diagonalizable.

The mean curvature equation is then written as
\begin{equation}
2yH-yy'x''+yx'y''- x'=0\,.\label{eqMeanCurvRevSurfSpaceAxisXYcurve}
\end{equation}
Since $\beta$ is parameterized by arc-length, we have the additional equation
\begin{equation}
x'\,^2+y'\,^2 = 1
{ \,\Rightarrow\, }
x'x'' = -y'y''.\label{eqAddRelationsRevSurfSpaceAxisXYcurve}
\end{equation}

Multiplying Eq. (\ref{eqMeanCurvRevSurfSpaceAxisXYcurve}) by $x'$ gives
\begin{equation}
2yx'H+(yy')'-1=0.
\end{equation}
On the other hand, multiplying Eq. (\ref{eqMeanCurvRevSurfSpaceAxisXYcurve}) by $y'$ gives
\begin{equation}
2yy'H-(yx')'=0.
\end{equation}
By defining $C(s)=y(s)y'(s)+{\rm i}\,y(s)x'(s)$ in $\mathbb{C}$ we can write the two equations above in a single expression as
\begin{equation}
C'(s)-2\,{\rm i}\,H(s) \,C(s)-1= 0\,.\label{eqMeanCurvWithComplexNumRevSurfSpaceAxisXYcurve}
\end{equation}

\subsection{Solution of the mean curvature equation for surfaces of revolution with a non-lightlike axis}

In this subsection we shall prove three theorems (Theorems \ref{thr::PrescMCRevSurfTaxisXZcurv}, \ref{thr::PrescMCRevSurfSaxisXZcurv}, and \ref{thr::PrescMCRevSurfSaxisXYcurv}) stating that,  given a continuous function $H:I\to\mathbb{R}$, there exists a 3-parameter family of $C^2$ curves whose corresponding {surface of revolution} has $C^0$ mean curvature $H$ when rotated around a non-lightlike axis as described in Figures 1(a), 1(b), and 1(c). Here, we also comment on the characterization of constant mean curvature surfaces of revolution as Delaunay surfaces (Theorem \ref{thr::DelaunayThr}).

For {surfaces of revolution} with a non-lightlike axis the mean curvature equation strongly depends on the causal character of the plane, $\Pi$, that contains the generating curve. Indeed, from Eqs. (\ref{eqMeanCurvWithHyperbolicNumRevSurfTimeAxisXZcurve}), (\ref{eqMeanCurvWithHyperbolicNumRevSurfSpaceAxisXZcurve}), and (\ref{eqMeanCurvWithComplexNumRevSurfSpaceAxisXYcurve}), the mean curvature equations relate to either
\begin{equation}
A'(s)+2\,\tau\,\eta\,H(s)\,A(s)-\eta=0
\mbox{ and }B'(s)+2\,\tau\,\eta\,H(s)\,B(s)+\eta=0
\end{equation}
if $\Pi$ is timelike or to
\begin{equation}
C'(s)-2\,\mathrm{i}\,H(s)\,C(s)-1=0 
\end{equation}
if $\Pi$ is spacelike. These equations can be  solved exactly:
\begin{enumerate}[(a)]
\item for a  curve on a timelike plane rotated around a timelike axis the solution is
\begin{equation}
A(s)=\left[\int_{0}^s\eta\,\mathrm{e}^{2\tau\eta\int_{0}^tH(u)\mathrm{d}u}\,\mathrm{d}t\right]\mathrm{e}^{-2\tau\eta\int_{0}^sH(t)\mathrm{d}t}+A_0\,\mathrm{e}^{-2\tau\eta\int_{0}^sH(t)\mathrm{d}t}\,,\label{EqSolutionMCeqForA}
\end{equation}
where $A_0$ is a constant; 
\item for a curve on a timelike plane rotated around a spacelike axis the solution is
\begin{equation}
B(s)=-\left[\int_{0}^s\eta\,\mathrm{e}^{2\tau\eta\int_{0}^tH(u)\mathrm{d}u}\,\mathrm{d}t\right]\mathrm{e}^{-2\tau\eta\int_{0}^sH(t)\mathrm{d}t}+B_0\,\mathrm{e}^{-2\tau\eta\int_{0}^sH(t)\mathrm{d}t}\,,\label{EqSolutionMCeqForB}
\end{equation}
where $B_0$ is a constant; and
\item for a curve on a spacelike plane (rotated around a spacelike axis) the solution is
\begin{equation}
C(s)=\left[\int_{0}^s\mathrm{e}^{-2\mathrm{i}\int_{0}^tH(u)\mathrm{d}u}\,\mathrm{d}t\right]\mathrm{e}^{2\mathrm{i}\int_{0}^sH(t)\mathrm{d}t}+C_0\,\mathrm{e}^{2\mathrm{i}\int_{0}^sH(t)\mathrm{d}t}\,,\label{EqSolutionMCeqForC}
\end{equation}
where $C_0$ is a constant.
\end{enumerate}

From the solutions above we can find a generating curve $(x(s),0,z(s))$ or $(x(s),y(s),0)$ leading to a surface with prescribed mean curvature $H$. Indeed, the Lorentzian variable $A$ in Eq. (\ref{eqMeanCurvWithHyperbolicNumRevSurfTimeAxisXZcurve}) satisfies
\begin{equation}
\left\{
\begin{array}{c}
A\bar{A} = \eta x^2\\
A-\bar{A} = 2\tau xz'\\
\end{array}
\right.\Rightarrow z' = \displaystyle\frac{A-\bar{A}}{2\tau\sqrt{\eta A\,\overline{A}}}\,.\label{eq::ExpressingzprimeUsingA}
\end{equation}
On the other hand, the Lorentzian variable $B$ in Eq. (\ref{eqMeanCurvWithHyperbolicNumRevSurfSpaceAxisXZcurve}) satisfies
\begin{equation}
\left\{
\begin{array}{c}
B\bar{B} = -\eta z^2\\
B-\bar{B} = 2\tau zx'\\
\end{array}
\right.\Rightarrow x' = \frac{B-\bar{B}}{2\tau\sqrt{-\eta B\,\overline{B}}}\,.
\end{equation}
Finally, the complex variable $C$ in Eq. (\ref{eqMeanCurvWithComplexNumRevSurfSpaceAxisXYcurve}) satisfies
\begin{equation}
\left\{
\begin{array}{c}
C\bar{C} = y^2\\
C-\bar{C} = 2\mathrm{i}yx'\\
\end{array}
\right.\Rightarrow x' = \frac{C-\bar{C}}{2\mathrm{i}\sqrt{C\,\overline{C}}}\,.\label{eq::ExpressingxprimeUsingC}
\end{equation}

\begin{theorem}
\label{thr::PrescMCRevSurfTaxisXZcurv}
Let $\alpha(s)=(x(s),0,z(s))$ be the generating curve of a $C^2$ {surface of revolution} with timelike axis $Oz$ and $C^0$ mean curvature $H(s)$. Then, we write $\alpha(s)$ as
\begin{equation}
\alpha(s;H,\mathbf{a})=(\sqrt{\eta[(g_1+\eta a_1)^2-(f_1+\eta a_2)^2]},0,\int_0^s\eta\frac{g_1'(f_1+\eta a_2)-f_1'(g_1+\eta a_1)}{\sqrt{\eta[(g_1+\eta a_1)^2-(f_1+\eta a_2)^2]}}\mathrm{d}t+a_3),\label{eq::CurveAlphaPrescH}
\end{equation}
where we have introduced the functions
\begin{equation}
\left\{
\begin{array}{c}
f_1(s)=\int_0^s\sinh(2\eta\int_0^tH(u)\mathrm{d}u)\mathrm{d}t\\[5pt]
g_1(s)=\int_0^s\cosh(2\eta\int_0^tH(u)\mathrm{d}u)\mathrm{d}t\\
\end{array}
\right.\label{eqAuxiliaryFunctionsMCeqUsingA}
\end{equation}
and the constant vector $\mathbf{a}=(a_1,a_2,a_3)$ satisfies the initial conditions at $s=0$ given by $\alpha(0)=(\sqrt{\eta(a_1^2-a_2^2)},0,a_3)$ and $\alpha'(0)=[\eta(a_1^2-a_2^2)]^{-1/2}(a_1,0,a_2)$.

Conversely, given a continuous function $H(s)$, $s\in I$, and a constant vector $(a_1,a_2,a_3)\in S_1\times\mathbb{R}$, where $S_1=\cup_{s\in I}\{(X,Y):[X+\eta g_1(s)]^2-[Y+\eta f_1(s)]^2\not=0\}$, then the $C^2$ curve $\alpha(s;H(s),\mathbf{a})$ generates a $C^2$ surface of revolution with $C^0$ mean curvature $H(s)$ when rotated around the (timelike) $z$-axis.
\end{theorem}
\begin{proof}
Using the functions introduced in Eq. (\ref{eqAuxiliaryFunctionsMCeqUsingA}), we can rewrite Eq. (\ref{EqSolutionMCeqForA}) as
\begin{equation}
A = \eta[(f_1+\eta a_2)+\tau(g_1+\eta a_1)](-f_1'+\tau g_1'),
\end{equation}
where $\tau A_0=a_2+\tau a_1$. Consequently,
\begin{equation}
\left\{
\begin{array}{c}
A-\bar{A} = 2\eta\tau\,[g_1'(f_1+\eta a_2)-f_1'(g_1+ \eta a_1)]\\
A\bar{A}=-(f_1+\eta a_2)^2+(g_1+\eta a_1)^2
\end{array}
\right.\,.
\end{equation}
Finally, inserting the relations above in Eq. (\ref{eq::ExpressingzprimeUsingA}) gives the expressions for $x(s)$ and $z(s)$ (after integration of $z'$) resulting in the expression for $\alpha(s;H(s),\mathbf{a})$ in Eq. (\ref{eq::CurveAlphaPrescH}). Geometrically, the constants $a_1,a_2,a_3$ are related to the initial conditions, i.e., position and initial velocity of the generating curve at $s=0$: $\alpha(0)=(\sqrt{\eta(a_1^2-a_2^2)},0,a_3)$ and $\alpha'(0)=[\eta(a_1^2-a_2^2)]^{-1/2}(a_1,0,a_2)$.

Conversely, given a continuous function $H:I\to\mathbb{R}$ and $(a_1,a_2,a_3)\in S_1\times \mathbb{R}$, notice that $S_1$ is a non-empty open subset of $\mathbb{R}^2$, the curve $\alpha(s;H(s),\mathbf{a})$ is a unit speed curve of class $C^2$ that generates a $C^2$ surface of revolution around the (timelike) $z$-axis with continuous mean curvature $H$.
\end{proof}

\begin{theorem}
\label{thr::PrescMCRevSurfSaxisXZcurv}
Let $\beta(s)=(x(s),0,z(s))$ be the generating curve of a $C^2$ surface of revolution with spacelike axis $Ox$ and $C^0$ mean curvature $H(s)$. Then, we write $\beta(s)$ as
\begin{equation}
\beta(s;H,\mathbf{b})=(\int_0^s-\eta\frac{g_2'(f_2-\eta b_1)-f_2'(g_2-\eta b_2)}{\sqrt{\eta[(f_2-\eta b_1)^2-(g_2-\eta b_2)^2]}}\mathrm{d}t+b_3,0,\sqrt{\eta[(f_2-\eta b_1)^2-(g_2-\eta b_2)^2]}),
\end{equation}
where we have introduced the functions
\begin{equation}
\left\{
\begin{array}{c}
f_2(s)=\int_0^s\sinh(2\eta\int_0^tH(u)\mathrm{d}u)\mathrm{d}t\\[5pt]
g_2(s)=\int_0^s\cosh(2\eta\int_0^tH(u)\mathrm{d}u)\mathrm{d}t\\
\end{array}
\right.\label{eqAuxiliaryFunctionsMCeqUsingB}
\end{equation}
and the constant vector $\mathbf{b}=(b_1,b_2,b_3)$ satisfies the initial conditions at $s=0$ given by $\beta(0)=(b_3,0,\sqrt{\eta(b_1^2-b_2^2)})$ and $\beta'(0)=[\eta(b_1^2-b_2^2)]^{-1/2}(b_1,0,b_2)$.

Conversely, given a continuous function $H(s)$, $s\in I$, and a constant vector $(b_1,b_2,b_3)\in S_2\times\mathbb{R}$, where $S_2=\cup_{s\in I}\{(X,Y):[X-\eta f_2(s)]^2-[Y-\eta g_2(s)]^2\not=0\}$, then the $C^2$ curve $\beta(s;H(s),\mathbf{b})$ generates a $C^2$ surface of revolution with $C^0$ mean curvature $H(s)$ when rotated around the (spacelike) $x$-axis.
\end{theorem}
\begin{proof}
The proof is analogous to the previous one, since the solution for the mean curvature equation in terms of $B$, Eq. (\ref{eqMeanCurvWithHyperbolicNumRevSurfSpaceAxisXZcurve}), is analogous to that of $A$, Eq. (\ref{eqMeanCurvWithHyperbolicNumRevSurfTimeAxisXZcurve}). Notice that here we should write the constant of integration $B_0$ in Eq. (\ref{EqSolutionMCeqForB}) as $\tau B_0=b_1+\tau\, b_2$. 
\end{proof}

\begin{theorem}
\label{thr::PrescMCRevSurfSaxisXYcurv}
Let $\gamma(s)=(x(s),y(s),0)$ be the generating curve of a $C^2$ surface of revolution with  spacelike axis $Ox$ and $C^0$ mean curvature $H(s)$. Then, we write $\gamma(s)$ as
\begin{equation}
\gamma(s;H(s),\mathbf{c})=(\int_0^s\frac{F'(G+c_2)-G'(F-c_1)}{\sqrt{(F-c_1)^2+(G+c_2)^2}}\mathrm{d}t+c_3\,,\,\sqrt{(F-c_1)^2+(G+c_2)^2},0),\label{eq::CurveGammaPrescH}
\end{equation}
where we have introduced the functions
\begin{equation}
\left\{
\begin{array}{c}
F(s)=\int_0^s\sin(2\int_0^tH(u)\mathrm{d}u)\mathrm{d}t\\[5pt]
G(s)=\int_0^s\cos(2\int_0^tH(u)\mathrm{d}u)\mathrm{d}t\\
\end{array}
\right.\label{eqAuxiliaryFunctionsMCeqUsingC}
\end{equation}
and the constant vector $\mathbf{c}=(c_1,c_2,c_3)$ satisfies the initial conditions at $s=0$ given by $\gamma(0)=(c_3,\sqrt{c_1^2+c_2^2},0)$ and $\gamma'(0)=(c_1^2+c_2^2)^{-1/2}(c_1,c_2,0)$.

Conversely, given a continuous function $H(s)$, $s\in I$, and constants $(c_1,c_2,c_3)\in T\times\mathbb{R}$, where $T=\cup_{s\in I}\{(X,Y):[X-F(s)]^2+[Y+G(s)]^2\not=0\}$, then the $C^2$ curve $\gamma(s;H(s),\mathbf{c})$ generates a $C^2$ surface of revolution with $C^0$ mean curvature $H(s)$ when rotated around the (spacelike) $x$-axis.
\end{theorem}
\begin{proof}
Using the functions introduced in Eq. (\ref{eqAuxiliaryFunctionsMCeqUsingC}), we can rewrite Eq. (\ref{EqSolutionMCeqForC}) as
\begin{equation}
C = [(F-c_1)+\mathrm{i}(G+c_2)](F'-\mathrm{i}G'),
\end{equation}
where $\mathrm{i}C_0=-c_1+\mathrm{i}c_2$. Consequently,
\begin{equation}
\left\{
\begin{array}{c}
C-\bar{C} = 2\mathrm{i}[F'(G+c_2)-G'(F-c_1)]\\
C\bar{C}=(F-c_1)^2+(G+c_2)^2
\end{array}
\right.\,.
\end{equation}
Finally, inserting the relations above in Eq. (\ref{eq::ExpressingxprimeUsingC}) gives the expressions for $x(s)$ (after integration of $x'$) and $y(s)$ resulting {in} the expression for $\gamma(s;H(s),\mathbf{c})$ in Eq. (\ref{eq::CurveGammaPrescH}). Geometrically, the constants $c_1,c_2,c_3$ are related to the initial conditions, i.e., position and initial velocity of the generating curve at $s=0$: $\gamma(0)=(c_3,\sqrt{c_1^2+c_2^2},0)$ and $\gamma'(0)=(c_1^2+c_2^2)^{-1/2}(c_1,c_2,0)$.

Conversely, given a continuous function $H:I\to\mathbb{R}$ and $(c_1,c_2,c_3)\in T\times \mathbb{R}$, notice that $T$ is a non-empty open subset of $\mathbb{R}^2$, the curve $\gamma(s;H(s),\mathbf{c})$ is a unit speed curve of class $C^2$ that generates a $C^2$ surface of revolution around the (spacelike) $x$-axis with continuous mean curvature $H$.
\end{proof}

\begin{remark}
Notice that the curvature function of the curves $(f_i,g_i)$, $i\in\{1,2\}$, and  $(F,G)$ are precisely $\kappa=2H$. Indeed, applying the expressions for the curvature function of a curve $c(s)$ in a Lorentzian $(+,-)$ and {in an} Euclidean $(+,+)$ plane, i.e., $\kappa=-\eta_{c'}\Vert c''\Vert$  \cite{daSilvaJG2017} and $\kappa=\Vert c''\Vert$, respectively, gives the desired result (here $\eta_{c'}=\langle c',c'\rangle=\sinh^2(2\eta\int H)-\cosh^2(2\eta\int H)=-1$). A similar result is valid in $\mathbb{E}^3$ \cite{KemmotsuTMJ1981}.
\end{remark}

To finish this subsection, let us mention that in Euclidean space a theorem due to Delaunay asserts that surfaces of revolution with constant mean curvature are precisely the undulary, nodary, and catenary \cite{DelaunayJMPA1841}. These surfaces are obtained by rotating roulettes of ellipses, hyperbolas, and parabolas, respectively \cite{EellsMI1987}. Delaunay-type theorems were already established for surfaces of revolution with generating curves on a timelike plane \cite{HanoTMJ1984,IshiharaJMTU1988}. In the next theorem, we shall prove that the same is valid for the situation where the generating curve lies on a spacelike plane.

\begin{theorem}[Delaunay-type theorem]
Let $\gamma(s)=(x(s),y(s),0)$ be the generating curve of a surface of revolution $S_{\gamma}$ with  spacelike axis $Ox$. Then, the surface $S_{\gamma}$ has constant mean curvature $H$ if, and only if, $\gamma$ is the roulette of a conic in the $xy$-plane.
\label{thr::DelaunayThr}
\end{theorem}
\begin{proof}
Since the $xy$-plane is spacelike, its conics and roulettes are the same as in the Euclidean plane. Finally, due to the fact that the mean curvature equations in Euclidean space, Eq. (1) of \cite{KemmotsuTMJ1981}, and  in Lorentz-Minkowski space (\ref{eqMeanCurvRevSurfSpaceAxisXYcurve}) are the same, it follows that constant mean curvature surfaces of revolution with generating curve in a spacelike axis are obtained from the revolution of roulettes of a conic.
\end{proof}

\subsection{Rotation of a curve on a timelike plane around a lightlike axis}

Let $\lambda:I\to \mathbb{E}_1^3$ be a $C^2$ regular curve in the $yz$-plane, i.e., $\lambda(s)=(0,y(s),z(s))$ with $s$ arc-length and $y>z$. Considering a rotation of this curve along a lightlike axis given by $(0,1,1)$ results in the following surface of revolution
\begin{equation}
L(s,\theta) = \Big([y(s)-z(z)]\,\theta,y(s)-\frac{\theta^2}{2}[y(s)-z(s)],z(s)-\frac{\theta^2}{2}[y(s)-z(s)]\Big),\label{defRevSurfLightAxisYZcurve}
\end{equation}
where $\theta\in(-\infty,+\infty)$. Since $s$ is the arc-length parameter of $\lambda$, we can write
\begin{equation}
y'\,^2-z'\,^2=\eta\in\{-1,+1\}\,.
\end{equation}

The first fundamental form is given by
\begin{equation}
\mathrm{I} = \eta\,{\rm d}s^2+(y-z)^2\,{\rm d}\theta^2\,.\label{eq1stFFormRevSurfLightAxisYZcurve}
\end{equation}
Since $g=g_{11}g_{22}-g_{12}^2=\eta(y-z)^2\Rightarrow\epsilon=-\eta$, it follows that $L$ is a spacelike (timelike) surface if and only if $\lambda(s)$ is a spacelike (timelike) curve.

Writing the normal vector to $L$ as
\begin{equation}
N = \Big(-(y'-z')\,\theta,z'+\frac{\theta^2}{2}(y'-z'),y'+\frac{\theta^2}{2}(y'-z')\Big),
\end{equation}
the second fundamental form is given by
\begin{equation}
\mathrm{II} = (y''z'-y'z'')\,{\rm d}s^2+(y-z)(y'-z')\,{\rm d}\theta^2\,.\label{eq2ndFFormRevSurfLightAxisYZcurve}
\end{equation}
Since both I and II are diagonal, the shape operator $A=\mathrm{I}^{-1}\mathrm{II}$ is diagonalizable.

The mean curvature equation is then written as
\begin{equation}
2(y-z)H+(y-z)(y''z'-y'z'')+\eta\,(y'-z')=0\,.\label{eqMeanCurvRevSurfLightAxisYZcurve}
\end{equation}
Since $\lambda(s)$ is parametrized by arc-length, we have the additional equation
\begin{equation}
y'\,^2-z'\,^2 = \eta 
{ \,\Rightarrow\, }
y'y'' = z'z''.\label{eqAddRelationsRevSurfLightAxisYZcurve}
\end{equation}

Multiplying Eq. (\ref{eqMeanCurvRevSurfLightAxisYZcurve}) by $(y'-z')$ gives
\begin{equation}
2\eta\,(y-z)(y'-z')H+[(y-z)(y'-z')]'=0.
\end{equation}
Solving the above equation gives $(y-z)(y'-z')=a_0\,\mathrm{e}^{-2\eta\int H}$. Then, we have the relation $(y-z)\mathrm{d}(y-z)=a_0\,\mathrm{e}^{-2\eta\int H}\,\mathrm{d}s$ and, consequently,
\begin{equation}
y(s)-z(s)=\Big\{a_1+a_0\,\int_{0}^s\exp[-2\eta\int_{0}^uH(t)\,{\rm d}t\,]\,{\rm d}u\Big\}^{1/2}\,,\label{eqSolutiony-zHLightAxis}
\end{equation}
where $a_0$ and $a_1$ are constants to be specified at $s=0$. On the other hand, multiplying Eq. (\ref{eqMeanCurvRevSurfLightAxisYZcurve}) by $(y'+z')$ gives
\begin{equation}
2(y-z)(y'+z')H-\eta(y-z)(y''+z'')+1=0\Rightarrow2\eta H\,\frac{y'+z'}{y-z}-\left(\frac{y'+z'}{y-z}\right)'=0.
\end{equation}

The solution of the above equation gives $y'+z'=(y-z)b_0\mathrm{e}^{2\eta\int H}$. Then, we have 
\begin{equation}
y(s)+z(s)=b_1+b_0\int_{0}^sa(u)\exp\Big(2\eta\int_{0}^uH(t){\rm d}t\Big){\rm d}u\,,
\end{equation}
where $b_0$, $b_1$ are constants and $a(u)=y(u)-z(u)$, Eq. (\ref{eqSolutiony-zHLightAxis}). Finally, from the knowledge of $y+z$ and $y-z$ we can find the expressions for $y$ and $z$:
\begin{equation}
y(s)=\frac{\left(b_1+\sqrt{a_1+a_0\int_{0}^s\mathrm{d}u\,\mathrm{e}^{2\eta\int_{0}^u\mathrm{d}t\,H}}+b_0\int_{0}^s\mathrm{d}u\,\mathrm{e}^{2\eta\int_{0}^u\mathrm{d}t\,H}\sqrt{a_1+a_0\int_{0}^u\mathrm{d}t\,\mathrm{e}^{2\eta\int_{0}^t\mathrm{d}v\,H}}\right)}{2};
\end{equation}
\begin{equation}
z(s)=-\frac{\left(b_1-\sqrt{a_1+a_0\int_{0}^s\mathrm{d}u\,\mathrm{e}^{2\eta\int_{0}^u\mathrm{d}t\,H}}+b_0\int_{0}^s\mathrm{d}u\,\mathrm{e}^{2\eta\int_{0}^u\mathrm{d}t\,H}\sqrt{a_1+a_0\int_{0}^u\mathrm{d}t\,\mathrm{e}^{2\eta\int_{0}^t\mathrm{d}v\,H}}\right)}{2}.
\end{equation}

\section{Prescribed skew curvature equation in Lorentz-Minkowski space}

We now address the problem of prescribed skew curvature for surfaces of revolution with a non-lightlike axis, as depicted in Figures 1(a), 1(b), and 1(c). Following da Silva \emph{et al.} \cite{DaSilvaAP2017}, the strategy consists in considering the generating curve as a graph and then write the equation for the skew curvature, which is initially a nonlinear second order ODE, as a linear first order ODE (with real coefficients). This approach can be seen as an adaptation of the techniques presented in \cite{BaikoussisJGeom} and \cite{BenekiJMAA2002} for helicoidal surfaces with prescribed mean/Gaussian curvature in Euclidean and Lorentz-Minkowski spaces, respectively. Unfortunately, we were not able to solve the $S$ prescribed problem for surfaces of revolution with a lightlike axis, Figure 1(d), with the present technique.

It is worth mentioning that in Ref. \cite{DaSilvaAP2017}, the authors point to the fact that a curve which is a graph in the $xz$-plane, say $\alpha(u)=(u,0,z(u))$, can be rotated around either the $x$- or $z$-axis. Nonetheless, these two possibilities lead to the same answer for the prescribed skew curvature problem in $\mathbb{E}^3$. Notice, however, that a priori these equivalent procedures do not make sense in $\mathbb{E}_1^3$, since distinct choices for the rotation axis lead to distinct types of surfaces (see Table 1). Instead, we should fix the axis and consider a curve as a graph in two ways (see subsections below).

\subsection{Rotation of a curve on a timelike plane around a timelike axis}

Let $\alpha:I\to \mathbb{E}_1^3$ be a $C^2$ regular curve in the $xz$-plane, i.e., $\alpha(u)=(x(u),0,z(u))$ with $x>0$. Considering a rotation of this curve around the timelike axis given by $(0,0,1)$ results in the following surface of revolution
\begin{equation}
Z(u,\theta) = (x(u)\cos\theta,x(u)\sin\theta,z(u)),
\end{equation}
where $\theta\in(0,2\pi)$. The causal character of $\alpha$ can be denoted through
\begin{equation}
\eta = \mathrm{sgn}(\langle\alpha',\alpha'\rangle) = \mathrm{sgn}(x'\,^2-z'\,^2)\in\{-1,1\}.
\end{equation}

The first fundamental form I is given by
\begin{equation}
\mathrm{I} = (x'\,^2-z'\,^2)\, {\rm d}u^2+x^2\,{\rm d}\theta^2\,.\label{eq1stFFormRevSurfTimeAxisXZGraphcurve}
\end{equation}
Since $\langle Z_u\times Z_{\theta},Z_u\times Z_{\theta}\rangle=-x^2(x'\,^2-z'\,^2)$, we have $\epsilon=-\eta$ and the normal to $Z(u,\theta)$ is
\begin{equation}
N = -\frac{1}{\sqrt{\eta(x'\,^2-z'\,^2)}}(z'\cos\theta,z'\sin\theta,x').
\end{equation}
The second fundamental form II is given by
\begin{equation}
\mathrm{II} = \frac{x'z''-x''z'}{\sqrt{\eta(x'\,^2-z'\,^2)}}\,{\rm d}u^2+\frac{xz'}{\sqrt{\eta(x'\,^2-z'\,^2)}}\,{\rm d}\theta^2\,.\label{eq2ndFFormRevSurfTimeAxisXZGraphcurve}
\end{equation}

\subsubsection{Generating curve as a graph with $x$ as independent variable}
Let $\alpha(u)=(u,0,z(u))$ be a graph with the $x$-direction as the independent variable. Here, the mean and Gaussian curvatures are
\begin{equation}
H = -\frac{uz''+z'(1-z'\,^2)}{2u[\eta(1-z'\,^2)]^{3/2}}\mbox{ and }K=-\frac{z'z''}{u(1-z'\,^2)^2}\,,
\end{equation}
respectively. Defining 
\begin{equation}
A=\frac{z'}{u\sqrt{\eta (1-z'\,^2)}}\mbox{ and } B=\frac{z'\,^2}{(1-z'\,^2)},\label{eq::AandBRevSurfTimeAxisXZGraphcurve}
\end{equation}
the Gaussian and mean curvatures  can be respectively written as the linear equations
\begin{equation}
A'+\frac{2}{u}A=-\eta\,\frac{2}{u}H\mbox{ and }B'=-2uK.\label{eq::HandKUsingAandBRevSurfTimeAxisXZGraphcurve}
\end{equation}
In addition, observing that $B=\eta\,u^2 A^2$,  we can write
\begin{equation}
S^2=H^2-\epsilon K = \left(A+\frac{u}{2}A'\right)^2-\left(A^2+uAA'\right) = \frac{u^2}{4}A'\,^2\,.
\end{equation}

The function $A(u)$ can be written in terms of the skew curvature $S$ as
\begin{equation}
A(u) = \pm\, 2\int \frac{S(v)}{v}\,\mathrm{d}v+a_0\,,
\end{equation}
where $a_0$ is a constant of integration. Now, using the expression for $A$ in Eq. (\ref{eq::AandBRevSurfTimeAxisXZGraphcurve}), we find that
\begin{equation}
z'\,^2=\frac{\eta\,u^2A^2}{1+\eta\,u^2A^2} \Rightarrow z(u) = \pm\int\frac{vA(v)}{\sqrt{\eta+v^2A^2(v)}}\,\mathrm{d}v+z_0\,,\label{EqGraphuZWithZaxis}
\end{equation}
where $z_0$ is another constant of integration.

\subsubsection{Generating curve as a graph with $z$ as independent variable}
Let $\alpha(u)=(x(u),0,u)$ be a graph with the $z$-direction as the independent variable. Here, the mean and Gaussian curvatures are
\begin{equation}
H = \frac{1-x'\,^2+xx''}{2x[\eta(x'\,^2-1)]^{3/2}}\mbox{ and }K=\frac{x''}{x(x'\,^2-1)^2}\,,
\end{equation}
respectively. Defining
\begin{equation}
A(u) = \frac{1}{x\sqrt{\eta(x'\,^2-1)}},
\end{equation}
we can write
\begin{equation}
\pm S = \frac{1-x'\,^2-xx''}{2x[\eta(x'\,^2-1)]^{3/2}}=\eta\,\frac{x}{2x'}\frac{\mathrm{d}A}{\mathrm{d}u}\,.
\end{equation}
Then, we have the following ODE for A
\begin{equation}
x\frac{\mathrm{d}A}{\mathrm{d}u}\mp 2\eta S\frac{\mathrm{d}x}{\mathrm{d}u}=\left(x\frac{\mathrm{d}A}{\mathrm{d}x}\mp 2\eta S\right)\frac{\mathrm{d}x}{\mathrm{d}u}=0\,.
\end{equation}
If $x'(u)\equiv0$ we have a cylinder. Otherwise, we are led to 
\begin{equation}
x\frac{\mathrm{d}A}{\mathrm{d}x}\mp 2\eta S=0\Rightarrow A=\pm2\eta\int\frac{S(v)}{v}\mathrm{d}v+a_1\,.
\end{equation}

Now, using the definition of $A$ above, we finally find that
\begin{equation}
\frac{\mathrm{d}x}{\mathrm{d}u}=\pm\sqrt{\frac{\eta+x^2A^2}{x^2A^2}}\Rightarrow u(x)=\pm\int\frac{vA(v)}{\sqrt{\eta+v^2A(v)^2}}\mathrm{d}v+u_0\,.
\end{equation}
Observe that the equation above is identical to Eq. (\ref{EqGraphuZWithZaxis}), but instead of finding $x(u)$ as a function of $u$ we found its inverse. This shows that a graph of a solution $f(u)$ of Eq. (\ref{EqGraphuZWithZaxis}) gives rise to a surface of revolution with prescribed $S$ with either $x$- or the $z$-axis as the independent variable direction, i.e., we can choose either $\alpha(u)=(u,0,f(u))$ or $\alpha(u)=(f(u),0,u)$ to rotate around the timelike $z$-axis. The only difference between these two choices lies in the causal character of $\alpha$.

\subsection{Rotation of a curve on a timelike plane around a spacelike axis}

Let $\beta:I\to \mathbb{E}_1^3$ be a $C^2$ regular curve in the $xz$-plane, i.e., $\beta(u)=(x(u),0,z(u))$ with $z>0$. Considering a rotation of this curve around the spacelike axis given by $(1,0,0)$ results in the following surface of revolution
\begin{equation}
X_I(u,\theta) = (x(u),z(u)\sinh\theta,z(u)\cosh\theta),
\end{equation}
where $\theta\in(-\infty,\infty)$. The causal character of $\beta$ can be described through
\begin{equation}
\eta = \mathrm{sgn}(\langle\beta',\beta'\rangle) = \mathrm{sgn}(x'\,^2-z'\,^2)\in\{-1,1\}.
\end{equation}

The first fundamental form I is given by
\begin{equation}
\mathrm{I} = (x'\,^2-z'\,^2)\, {\rm d}u^2+z^2\,{\rm d}\theta^2\,.\label{eq1stFFormRevSurfSpaceAxisXZGraphcurve}
\end{equation}
Since $\langle \partial_uX_I\times \partial_{\theta}X_I,\partial_uX_I\times \partial_{\theta}X_I\rangle=-z^2(x'\,^2-z'\,^2)$, we have $\epsilon=-\eta$ and the normal to $X_I(u,\theta)$ is
\begin{equation}
N = -\frac{1}{\sqrt{\eta(x'\,^2-z'\,^2)}}(z',x'\sinh\theta,x'\cosh\theta).
\end{equation}
The second fundamental form II is given by
\begin{equation}
\mathrm{II} = \frac{x'z''-x''z'}{\sqrt{\eta(x'\,^2-z'\,^2)}}\,{\rm d}u^2+\frac{x'z}{\sqrt{\eta(x'\,^2-z'\,^2)}}\,{\rm d}\theta^2\,.\label{eq2ndFFormRevSurfSpaceAxisXZGraphcurve}
\end{equation}

\subsubsection{Generating curve as a graph with $z$ as independent variable}
Let $\beta(u)=(x(u),0,u)$ be a graph with the $z$-direction as the independent variable. Here, the mean and Gaussian curvatures are
\begin{equation}
H = \frac{ux''-x'(x'\,^2-1)}{2u[\eta(x'\,^2-1)]^{3/2}}\mbox{ and }K=\frac{x'x''}{u(x'\,^2-1)^2}\,,
\end{equation}
respectively. Defining 
\begin{equation}
A=\frac{x'}{u\sqrt{\eta (x'\,^2-1)}}\mbox{ and } B=\frac{x'\,^2}{(x'\,^2-1)}\,,\label{eq::AandBRevSurfSpaceAxisXZGraphcurve}
\end{equation}
the Gaussian and mean curvatures can be respectively written as the linear equations
\begin{equation}
A'+\frac{2}{u}A=-\eta\,\frac{2}{u}H\mbox{ and }B'=-2uK\,.\label{eq::HandKUsingAandBRevSurfSpaceAxisXZGraphcurve}
\end{equation}

Observing that the equations above are analogous to those of $Z$, Eqs. (\ref{eq::AandBRevSurfTimeAxisXZGraphcurve}) and (\ref{eq::HandKUsingAandBRevSurfTimeAxisXZGraphcurve}), we have
\begin{equation}
A(u) = \pm\, 2\int \frac{S(v)}{v}\,\mathrm{d}v+a_0\,,
\end{equation}
where $a_0$ is a constant of integration. Now, using the expression for $A$ in Eq. (\ref{eq::AandBRevSurfSpaceAxisXZGraphcurve}), we find that
\begin{equation}
x'\,^2=-\frac{\eta\,u^2A^2}{1-\eta\,u^2A^2}\Rightarrow x(u) = \pm\int\frac{vA(v)}{\sqrt{-\eta+v^2A^2(v)}}\,\mathrm{d}v+x_0\,,\label{EqGraphuZWithXaxis}
\end{equation}
where $x_0$ is another constant of integration. 

\subsubsection{Generating curve as a graph with $x$ as independent variable}
Let $\beta(u)=(u,0,z(u))$ be a graph with the $x$-direction as the independent variable. Here, the mean and Gaussian curvatures are
\begin{equation}
H = -\frac{1-z'\,^2+zz''}{2z[\eta(1-z'\,^2)]^{3/2}}\mbox{ and }K=-\frac{z''}{z(1-z'\,^2)^2}\,,
\end{equation}
respectively. Defining
\begin{equation}
A(u) = \frac{1}{z\sqrt{\eta(1-z'\,^2)}},
\end{equation}
we can write
\begin{equation}
\pm S = \frac{1-z'\,^2-zz''}{2z[\eta(1-z'\,^2)]^{3/2}}=-\eta\,\frac{z}{2z'}\frac{\mathrm{d}A}{\mathrm{d}u}\,.
\end{equation}
Then, we have the following ODE for A
\begin{equation}
z\frac{\mathrm{d}A}{\mathrm{d}u}\pm 2\eta S\frac{\mathrm{d}z}{\mathrm{d}u}=\left(z\frac{\mathrm{d}A}{\mathrm{d}z}\pm 2\eta S\right)\frac{\mathrm{d}z}{\mathrm{d}u}=0\,.
\end{equation}
If $z'(u)\equiv0$ we have a cylinder. Otherwise, we are led to 
\begin{equation}
z\frac{\mathrm{d}A}{\mathrm{d}z}\pm 2\eta S=0\Rightarrow A=\pm2\eta\int\frac{S(v)}{v}\mathrm{d}v+a_1\,.
\end{equation}

Now, using the definition of $A$ above, we finally find that
\begin{equation}
\frac{\mathrm{d}z}{\mathrm{d}u}=\pm\sqrt{\frac{-\eta+z^2A^2}{z^2A^2}}\Rightarrow u(z)=\pm\int\frac{vA(v)}{\sqrt{-\eta+v^2A(v)^2}}\mathrm{d}v+u_0\,.
\end{equation}
Observe that the equation above is identical to Eq. (\ref{EqGraphuZWithXaxis}), but instead of finding $z(u)$ as a function of $u$ we found its inverse. This shows that a graph of a solution $f(u)$ of Eq. (\ref{EqGraphuZWithXaxis}) gives rise to a surface of revolution with prescribed $S$ with either $x$- or the $z$-axis as the independent variable direction, i.e., we can choose either $\beta(u)=(u,0,f(u))$ or $\beta(u)=(f(u),0,u)$ to rotate around the spacelike $x$-axis. The only difference between these two choices lies in the causal character of $\beta$.

\subsection{Rotation of a curve on a spacelike plane around a spacelike axis}
Let $\gamma:I\to \mathbb{E}_1^3$ be a $C^2$ regular curve in the $xy$-plane, i.e., $\gamma(u)=(x(u),y(u),0)$ with $y>0$. Considering a rotation of this curve around the (spacelike) $x$-axis results in the following surface of revolution
\begin{equation}
X_{II}(u,\theta) = (x(u),y(u)\cosh\theta,y(u)\sinh\theta),
\end{equation}
where $\theta\in(-\infty,\infty)$. The curve $\gamma$ is always spacelike, since
\begin{equation}
\langle\gamma',\gamma'\rangle=(x'\,^2+y'\,^2)>0.
\end{equation}

The first fundamental form I is given by
\begin{equation}
\mathrm{I} = (x'\,^2+y'\,^2)\, {\rm d}u^2-y^2\,{\rm d}\theta^2\,,\label{eq1stFFormRevSurfSpaceAxisXYGraphcurve}
\end{equation}
the normal to $X_{II}(u,\theta)$ is
\begin{equation}
N = \frac{1}{\sqrt{x'\,^2+y'\,^2}}(y',-x'\cosh\theta,-x'\sinh\theta),
\end{equation}
and the second fundamental form II is
\begin{equation}
\mathrm{II} = \frac{x''y'-x'y''}{\sqrt{x'\,^2+y'\,^2}}\,{\rm d}u^2-\frac{x'y}{\sqrt{x'\,^2+y'\,^2}}\,{\rm d}\theta^2\,.\label{eq2ndFFormRevSurfSpaceAxisXYGraphcurve}
\end{equation}

\subsubsection{Generating curve as a graph with $y$ as independent variable}
Let $\gamma(u)=(x(u),u,0)$ be a graph with the $y$-direction as the independent variable. Here, the mean and Gaussian curvatures are
\begin{equation}
H = \frac{ux''+x'(1+x'\,^2)}{2u(1+x'\,^2)^{3/2}}\mbox{ and }K=\frac{x'x''}{u(1+x'\,^2)^2}\,,
\end{equation}
respectively. Defining 
\begin{equation}
A=\frac{x'}{u\sqrt{(1+x'\,^2)}}\mbox{ and } B=\frac{x'\,^2}{(1+x'\,^2)}\,,\label{eq::AandBRevSurfSpaceAxisXYGraphcurve}
\end{equation}
the Gaussian and mean curvatures can be respectively written as the linear equations
\begin{equation}
A'+\frac{2}{u}A=\frac{2}{u}H\mbox{ and }B'=2uK\,.\label{eq::HandKUsingAandBRevSurfSpaceAxisXYGraphcurve}
\end{equation}

In addition, observing that $B=u^2 A^2$,  we can write
\begin{equation}
S^2=H^2-K = \left(A+\frac{u}{2}A'\right)^2-\left(A^2+uAA'\right) = \frac{u^2}{4}A'\,^2\,.
\end{equation}

The function $A(u)$ can be written in terms of the skew curvature $S$ as
\begin{equation}
A(u) = \pm\, 2\int \frac{S(v)}{v}\,\mathrm{d}v+a_0\,,
\end{equation}
where $a_0$ is a constant of integration. Now, using the expression for $A$ in Eq. (\ref{eq::AandBRevSurfSpaceAxisXYGraphcurve}), we find that
\begin{equation}
x'\,^2=\frac{u^2A^2}{1-u^2A^2} \Rightarrow x(u) = \pm\int\frac{vA(v)}{\sqrt{1-v^2A^2(v)}}\,\mathrm{d}v+x_0\,,\label{EqGraphuYWithYaxis}
\end{equation}
where $x_0$ is another constant of integration.

\subsubsection{Generating curve as a graph with $x$ as independent variable}
Let $\gamma(u)=(u,y(u),0)$ be a graph with the $x$-direction as the independent variable. Here, the mean and Gaussian curvatures are
\begin{equation}
H = \frac{1+y'\,^2-yy''}{2y(1+y'\,^2)^{3/2}}\mbox{ and }K=-\frac{y''}{y(1+y'\,^2)^2}\,,
\end{equation}
respectively. Defining
\begin{equation}
A(u) = \frac{1}{y\sqrt{1+y'\,^2}},
\end{equation}
we can write
\begin{equation}
\pm S = \frac{1+y'\,^2+yy''}{2y(1+y'\,^2)^{3/2}}=-\frac{y}{2y'}\frac{\mathrm{d}A}{\mathrm{d}u}\,.
\end{equation}
Then, we have the following ODE for A
\begin{equation}
y\frac{\mathrm{d}A}{\mathrm{d}u}\pm 2 S\frac{\mathrm{d}y}{\mathrm{d}u}=\left(y\frac{\mathrm{d}A}{\mathrm{d}y}\pm 2 S\right)\frac{\mathrm{d}y}{\mathrm{d}u}=0\,.
\end{equation}
If $y'(u)\equiv0$ we have a cylinder. Otherwise, we are led to 
\begin{equation}
y\frac{\mathrm{d}A}{\mathrm{d}y}\pm 2 S=0\Rightarrow A(u)=\pm2\int\frac{S(v)}{v}\mathrm{d}v+a_1\,.
\end{equation}

Now, using the definition of $A$ above, we finally find that
\begin{equation}
\frac{\mathrm{d}y}{\mathrm{d}u}=\pm\sqrt{\frac{1-y^2A^2}{y^2A^2}}\Rightarrow u(y)=\pm\int\frac{vA(v)}{\sqrt{1-v^2A(v)^2}}\mathrm{d}v+u_0\,.
\end{equation}
Observe that the equation above is identical to Eq. (\ref{EqGraphuYWithYaxis}), but instead of finding $y(u)$ as a function of $u$ we found its inverse. This shows that a graph of a solution $f(u)$ of Eq. (\ref{EqGraphuYWithYaxis}) gives rise to a surface of revolution with prescribed $S$ with either $x$- or the $y$-axis as the independent variable direction, i.e., we can choose either $\gamma(u)=(u,f(u),0)$ or $\gamma(u)=(f(u),u,0)$ to rotate around the spacelike $x$-axis.  Notice that the causal character of $\gamma$, and the respective surface of revolution, does not depend on this choice.

\appendix

\section{Lorentz numbers}

The ring of \emph{Lorentz numbers} $\mathbb{L}$, often named double or hyperbolic numbers \cite{Yaglom1979}, is $\mathbb{L}=\{a+b\tau:a,b\in\mathbb{R},\tau\not\in\mathbb{R},\mbox{ and }\tau^2=1\}$, which is isomorphic to $\mathbb{R}[X]/(X^2-1)$, where $\mathbb{R}[X]$ is the ring of real polynomials. The sum and product in the (commutative) ring $\mathbb{L}$ is defined as usual: $(a+b\tau)+(\alpha+\beta\tau)=(a+\alpha)+(b+\beta)\tau$ and $(a+b\tau)(\alpha+\beta\tau)=(a\alpha+b\beta)+(a\beta+b\alpha)\tau$. Moreover, $\mathbb{L}$ does not form a field, since $(a\pm a\tau)^2=0$ {for all $a$}.

It is also possible to define a conjugation of $w=a+b\tau\in\mathbb{L}$ as usual $\bar{w}=a-b\tau$. Consequently, if $a^2-b^2\not=0$, then $w^{-1}=\bar{w}/(w\bar{w})$, where $w\bar{w}=a^2-b^2\in\mathbb{R}$. A Lorentz number $w$ is \emph{space-}, \emph{time-}, or \emph{light-like} if $w\bar{w}>0$ or $w=0$, $w\bar{w}<0$, or $w\bar{w}=0$, respectively. The lightlike numbers are precisely the zero divisors of $\mathbb{L}$ and are of the form $a\pm a\tau$. The set of invertible Lorentz numbers is $\mathbb{L}^{^*}=\{w\in\mathbb{L}:\exists\,w^{-1}\}=\mathbb{L}-{\{a\pm a\tau\}}$.

In addition, the Lorentz numbers admit the linear representation
\begin{equation}
a+b\tau \mapsto \left(\begin{array}{cc}
a & b\\
b & a\\
\end{array}\right)\,,
\end{equation}
from which we can define a polar form of a Lorentz number $w\in\mathbb{L}^{^*}$ to be 
\begin{equation}
a+b\tau=\left\{
\begin{array}{c}
r(\cosh\,\theta+\tau\sinh\,\theta),\mbox{ if }a^2-b^2>0\\
r(\sinh\,\theta+\tau\cosh\,\theta),\mbox{ if }a^2-b^2<0\\
\end{array}
\right.\,,
\end{equation}
where $r=\sqrt{\vert w\bar{w}\vert}=\sqrt{\vert a^2-b^2\vert}$ is the length of $w$. We also define an 
exponential function
\begin{equation}
\exp(a+b\tau)=\mathrm{e}^{a+b\tau}=
\mathrm{e}^a(\cosh\,b+\tau\sinh\,b)\,.
\end{equation}

Given a function $f(s)=a(s)+b(s)\tau$, where $a,b$ are differentiable real functions of a real variable $s$, it is easy to verify using the expressions above that
\begin{equation}
\frac{\mathrm{d}}{\mathrm{d}s}\mathrm{e}^{f(s)}=[a'(s)+b'(s)\tau]\,\mathrm{e}^{f(s)}=f'(s)\mathrm{e}^{f(s)}\,.
\end{equation}
\begin{example}
The solution of the linear ODE $\mathrm{d}w/\mathrm{d}t+g(s)\tau\, w(s)=0$ with initial condition $w(s_0)=w_0$, where $s$ and $g(s)$ are real, is given by
\begin{equation}
\label{eq::SolutionLorentzianHomogenODE}
w(s)=w_0\,\exp\left(\tau\displaystyle\int_{s_0}^sg(u)\mathrm{d}u\right)\,.
\end{equation}
Notice that this ODE is equivalent to the system
\begin{equation}
\left\{\begin{array}{c}
x'+g(s)y=0\\
y'+g(s)x=0\\
\end{array}
\right.,\,x(s_0)=\mbox{Re}(w_0),\,y(s_0)=\mbox{Im}(w_0).
\end{equation}
\end{example}

\begin{remark}
It is possible to define a notion of differentiability for  functions $f:\mathbb{L}\to\mathbb{L}$ as done, e.g., in Ref. \cite{KonderakDGDS2014}. {(In fact, it is possible to introduce a notion of differentiability for functions over any algebra, see e.g. \cite{GadeaPAMS1996}.)} However, the few concepts and formalism introduced in this Appendix suffice for our purposes. 
\end{remark}

\end{document}